\documentclass[a4paper,11pt]{article}
\usepackage[centertags]{amsmath}
\usepackage{amsfonts}
\usepackage{amssymb}
\usepackage{amsthm}
 \newtheorem{theorem}{Theorem}[section]
 \newtheorem{corollary}[theorem]{Corollary}
 \newtheorem{lemma}[theorem]{Lemma}
 
 \theoremstyle{definition}
 \newtheorem{definition}[theorem]{Definition}
 \theoremstyle{remark}
 \newtheorem{remark}[theorem]{Remark}
 \newtheorem{example}[theorem]{Example}
 \numberwithin{equation}{section}
\begin{document}
\begin{center}\Large{\textbf{On graded nil clean rings}}
\footnote{This is a preprint of a paper published in Communications in Algebra, 46 (9) (2018), 4079--4089.\\
https://doi.org/10.1080/00927872.2018.1435791\\
https://www.tandfonline.com/loi/lagb20}
\let\thefootnote\relax
\end{center}
\begin{center}\textbf{Emil Ili\'{c}-Georgijevi\'{c} and Serap 
\c{S}ahinkaya}\end{center}
\begin{abstract}
\noindent In this paper we introduce and study the notion of a graded (strongly) 
nil clean ring which is group graded. We also deal with extensions of graded
(strongly) nil clean rings to graded matrix rings and to graded group rings. The 
question of when nil cleanness of the component, which corresponds to the 
neutral
element of a group, implies graded nil cleanness of the whole graded ring is
examined. Similar question is discussed in the case of groupoid graded rings as 
well. 
\end{abstract}
\section{Introduction}
\noindent Since the introduction of \emph{clean rings} in \cite{wkn} as rings in
which every element can be written as a sum of an idempotent element and a unit
element, many authors have investigated rings in which
elements can be written as a sum of an idempotent element and an element with a
certain property. Recently, in \cite{ajd}, the notion of a (\emph{strongly})
\emph{nil clean ring} was introduced as a ring in which every element can be 
written as
a sum of an idempotent element and a nilpotent element (such that an idempotent
element and a nilpotent element mutually commute). Many significant results
concerning extensions of such rings to matrix rings have been obtained 
(see \cite{bcdm,klz,kwz}) which are 
related to the famous K\"{o}the's Conjecture (see \cite{jm}). Also, in 
\cite{mrs}, an
extension to group rings in commutative case is investigated, while in 
\cite{stz} this
is generalized to noncommutative case. On the other hand, theory of graded rings
has been studied by many authors (see \cite{ak,noy}). The aim of this paper is to 
introduce the graded ring theory into the study of the above mentioned ring 
element properties. Following \cite{ajd}, we introduce the notion of a \emph{graded} 
(\emph{strongly}) \emph{$\mathcal{P}$-clean ring}, where $\mathcal{P}$ is a
\emph{graded ABAB-compatible property} of a homogeneous element of a graded
ring. In particular, we study \emph{graded clean} and \emph{graded nil clean rings}.
However, emphasis is on graded (strongly) nil clean rings.\\ 
\indent After establishing some characterizations and
basic properties of graded (strongly) nil clean rings, we focus on extensions of
graded clean and graded nil clean rings to
graded matrix rings thus generalizing results from \cite{hn} and \cite{ajd}. As
already mentioned, in \cite{stz} nil clean group rings are investigated. Here we 
extend some of the results to graded group rings case. This yields an interesting
question on how the graded nil cleanness of a group graded ring depends
on the nil cleanness of the component which corresponds to the neutral element of
a group. We also take a look at a similar question in the case of rings graded 
by a partial groupoid (see \cite{ak1,ak2}).  
\section{Preliminaries}
\noindent All rings are assumed to be associative with identity. For more details on 
everything stated in the most of this section, we refer to \cite{ak,noy}.\\ 
\indent Let $R$ be a ring, $G$ a group with the identity element $e,$ and let
$\{R_g\}_{g\in G}$ be a family of additive subgroups of $R.$
Recall that $R$ is said to be $G$-\emph{graded} if $R=\bigoplus_{g\in
G}R_g$ and $R_gR_h\subseteq R_{gh}$ for all $g,h\in G.$ The set
$H=\bigcup_{g\in G}R_g$ is called the \emph{homogeneous part} of
$R,$ elements of $H$ are called \emph{homogeneous}, and subgroups 
$R_g$ $(g\in G)$ are called \emph{components}. If $a\in R_g,$ then we say that
$a$ has the \emph{degree} $g.$\\
\indent In the most of this paper we work in the category whose objects are 
$G$-graded
rings and morphisms are homomorphisms of rings which are degree-preserving.\\
\indent A right ideal 
(left, two-sided) $I$ of a graded ring $R=\bigoplus_{g\in G}R_g$ is called
\emph{homogeneous} or \emph{graded} if $I=\bigoplus_{g\in G}I\cap
R_g.$ If $I$ is a two-sided homogeneous ideal (homogeneous ideal in the rest of
the paper), then $R/I$ is a $G$-graded ring with components 
$(R/I)_g=R_g/I\cap R_g.$
A graded ring $R$ is \emph{graded-nil} if every homogeneous
element of $R$ is nilpotent. We also know that $1\in R_e.$\\ 
\indent As told in the previous section, we extend some results from \cite{stz} 
to the
graded case. In order to do that, we recall a way to make a group ring graded. Let 
$R=\bigoplus_{g\in G}R_g$ be a $G$-graded ring, and observe the group ring 
$R[G].$ Then, according to \cite{cn}, $R[G]$ is $G$-graded (actually, strongly
graded) with the $g$-th component $(R[G])_g=\sum_{h\in G}R_{gh^{-1}}h$ and
with the multiplication defined via the rule 
$(r_gg')(r_hh')=r_gr_h(h^{-1}g'hh'),$
where $g,g',h,h'\in G$ and $r_g\in R_g,$ $r_h\in R_h.$\\
\indent If $H$ is a normal subgroup of $G,$ then, according to \cite{noy}, we may
observe $R[H]$ as a $G$-graded ring $\bigoplus_{g\in G}(R[H])_g,$ where
$(R[H])_g=\bigoplus_{h\in H}R_{gh^{-1}}h,$ and where the multiplication is given 
by $(r_gg')(r_hh')=r_gr_h(h^{-1}g'hh'),$
where $g,h\in G,$ $g',h'\in H$ and $r_g\in R_g,$ $r_h\in R_h.$\\
\indent If $R=\bigoplus_{g\in G}R_g$ is a $G$-graded ring, then a right $G$-graded 
$R$-module is a right $R$-module $M$ such that $M=\bigoplus_{x\in G}M_x,$ 
where $M_x$ are additive subgroups of $M,$ and such that 
$M_xR_g\subseteq M_{xg}$ for all $x, g\in G.$ A submodule $N$ of a $G$-graded
$R$-module $M=\bigoplus_{x\in G}M_x$ is called
\emph{homogeneous} if $N=\bigoplus_{x\in G}N\cap M_x.$
The category whose objects are right $G$-graded $R$-modules 
and morphisms are homomorphisms which are degree-preserving 
is denoted by $R-gr.$\\ 
\indent Let $M=\bigoplus_{x\in G}M_x\in R-gr.$ If $\mathrm{END}_R(M)_g$ 
consists of
endomorphisms $f:M\to M$ such that $f(M_x)\subseteq M_{gx}$ for all $x\in G,$
then $\mathrm{END}_R(M)=\bigoplus_{g\in G}\mathrm{END}_R(M)_g$
is a $G$-graded ring with respect to the usual addition and multiplication defined by
$f\cdot g=f\circ g$ $(f, g\in\mathrm{END}_R(M)).$\\
\indent If $R$ is a $G$-graded ring and $n$ a natural number, then we know
that the matrix ring $M_n(R)$ over $R$ can be made into a $G$-graded ring in the
following way. Let $\overline{\sigma}=(g_1,\dots,g_n)\in G^n,$ $\lambda\in G$ 
and 
\[M_n(R)_\lambda(\overline{\sigma})=\left(%
\begin{array}{cccc}
  R_{g_1\lambda g_1^{-1}} & R_{g_1\lambda g_2^{-1}} & \dots & R_{g_1\lambda 
g_n^{-1}} \\
  R_{g_2\lambda g_1^{-1}} & R_{g_2\lambda g_2^{-1}} & \dots & R_{g_2\lambda 
g_n^{-1}} \\
  \vdots & \vdots & \dots & \vdots \\
  R_{g_n\lambda g_1^{-1}} & R_{g_n\lambda g_2^{-1}} & \dots & R_{g_n\lambda 
g_n^{-1}} \\
\end{array}%
\right).\]
Then $M_n(R)=\bigoplus_{\lambda\in G}M_n(R)_\lambda(\overline{\sigma})$ is a
$G$-graded ring with respect to the usual matrix addition and multiplication. Usually, 
this ring is denoted by $M_n(R)(\overline{\sigma}).$\\
\indent A graded module $M\in R-gr$ is said to be \emph{graded simple} 
(or \emph{graded irreducible}) if $MR\neq0$ and if the only homogeneous
submodules of 
$M$ are trivial submodules. The \emph{graded Jacobson radical} $J^g(R)$ of a 
$G$-graded ring $R$ is defined to be the intersection of annihilators of all 
graded simple graded $R$-modules. It is known that $J^g(R)$ coincides with the
intersection of all maximal homogeneous right ideals of $R,$ and that it is left-right
symmetric. As usual, $J(R)$ denotes the classical Jacobson radical of $R.$\\
\indent We also recall the notion of a graded ring graded in the following sense.
\begin{definition}[\cite{ak2,ak,ak1}]\label{sgris} Let $R$
be a ring, and $S$ a partial groupoid, that is, a set with a
partial binary operation. Also, let $\{R_s\}_{s\in
S}$ be a family of additive subgroups of $R.$ We say that
$R=\bigoplus_{s\in S}R_s$ is $S$-\emph{graded}
and $R$ \emph{induces} $S$ (or $R$ is an $S$-\emph{graded
ring inducing} $S$) if the following two conditions hold:
\begin{itemize}
    \item[$i)$] $R_s R_t\subseteq R_{st}$ whenever $st$ is 
defined;
    \item[$ii)$] $R_s R_t\neq0$ implies that the product $st$ is
    defined.
\end{itemize}
\end{definition} 
\indent We refer to \cite{ak} for more information concerning $S$-graded rings
inducing $S$ (for taking quotients by homogeneous ideals, one may also consult
\cite{ak3}).\\
\indent The notion of a graded ring presented in \cite{ha1}, as 
well as in \cite{cha,ha,agg}, is equivalent to that from Definition~\ref{sgris}.
There it is studied from the homogeneous point of view by observing the
homogeneous part of a graded ring with induced partial addition and everywhere
defined multiplication. Such a structure
is named \emph{anneid} \cite{cha,ha,agg} and an origin of this approach goes 
back to \cite{mk}. For a survey on anneids, one may also consult \cite{kv,mv}. 
Everything obtained for anneids holds for $S$-graded rings inducing $S$ and vice
versa.\\
\indent The role of a 
degree-preserving homomorphism is taken by the following notion. If $R$ is an 
$S$-graded
ring inducing $S$ and $R'$ is an $S'$-graded ring inducing $S',$ then a ring
homomorphism $f:R\to R'$ is called \emph{homogeneous} \cite{cha,ha,agg} if a
homogeneous element is mapped to a homogeneous element, and if the fact that
$f(x)$ is a nonzero homogeneous element of $R'$ implies that $x$ is a
homogeneous element of $R.$ The corresponding notion for anneids is simply 
\emph{homomorphism of anneids}.
\section{Graded nil clean rings}
Let $G$ be a group with the identity element $e.$
\begin{definition}
A homogeneous element $r$ of a $G$-graded ring $R$ is
called \emph{graded nil clean} (\emph{graded strongly nil clean}) if it can be 
written
as a sum of a homogeneous idempotent element $f$ and a homogeneous nilpotent
element $b$ (such that $fb=bf$). A $G$-graded ring is called 
\emph{graded nil clean} (\emph{graded strongly nil clean}) if every of its
homogeneous elements is graded nil clean (graded strongly nil clean).
\end{definition}
\begin{remark}\label{remark}
Let a $G$-graded ring $R=\bigoplus_{g\in G}R_g$ be graded (strongly) nil clean. 
If a
homogeneous idempotent is nonzero, it has to be from $R_e,$ of course. If 
$x\in R_g$ is a nonzero element, then either it is nilpotent or $g=e$ and 
$x=y+z,$
where $y$ is an idempotent element and $z$ a nilpotent element from $R_e$ (such
that $yz=zy$), that is, $R_e$ is a (strongly) nil clean ring. Obviously, every 
graded-nil ring is strongly nil clean.
\end{remark}
\begin{example}\label{example}
Let $S$ be a nil clean ring. Then the ring of matrices $R=\left(%
\begin{array}{cc}
  S & S \\
  S & S \\
\end{array}%
\right)$ is a graded nil clean ring with respect to a well known 
$\mathbb{Z}$-grading
$R_0=\left(%
\begin{array}{cc}
  S & 0 \\
  0 & S \\
\end{array}%
\right),$ $R_1=\left(%
\begin{array}{cc}
  0 & S \\
  0 & 0 \\
\end{array}%
\right),$ $R_{-1}=\left(%
\begin{array}{cc}
  0 & 0 \\
  S & 0 \\
\end{array}%
\right),$ $R_i=\left(%
\begin{array}{cc}
  0 & 0 \\
  0 & 0 \\
\end{array}%
\right)$ for $i\in\mathbb{Z}\setminus\{-1,0,1\}.$
\end{example}
\indent We may introduce a more general notion following the notion of a 
\emph{clean $\mathcal{P}$-ring} from \cite{ajd}, where $\mathcal{P}$ is the
so-called \emph{ABAB-compatible property} that an element of a ring may satisfy.
If $R=\bigoplus_{g\in G}R_g$ is a $G$-graded ring, and $f$ an idempotent element 
of $R_e,$ then by a \emph{graded corner ring} we mean the $G$-graded ring
$fRf=\bigoplus_{g\in G}fR_gf.$
\begin{definition}
Let $\mathcal{P}$ be a property that a homogeneous element of a $G$-graded ring 
$R$ can satisfy. We say that $\mathcal{P}$ is a \emph{graded ABAB-compatible
property} if it satisfies the following conditions:
\begin{itemize}
\item[$i)$] If $a$ is a homogeneous element of $R$ which has property 
$\mathcal{P},$ then $-a$ also has property $\mathcal{P};$
\item[$ii)$] If a homogeneous element $a\in R$ has property $\mathcal{P}$ and
$f$ is a homogeneous idempotent element of $R$ such that $af=fa,$ then
$faf\in fRf$ has property $\mathcal{P};$
\item[$iii)$] If $a$ is a homogeneous element of $R$ and $f$ a homogeneous 
idempotent element of $R$ such that $af=fa,$ and if the elements $faf\in fRf$ and
$(1-f)a(1-f)\in(1-f)R(1-f)$ both have property $\mathcal{P}$ as elements of
the respective graded corner rings, then $a$ has property $\mathcal{P}$ in $R.$ 
\end{itemize}
A $G$-graded ring is said to be \emph{graded $\mathcal{P}$-clean} if every of its 
homogeneous elements can be written as a sum of a homogeneous idempotent
element and a homogeneous element with graded ABAB-compatible property 
$\mathcal{P}.$
\end{definition}
\indent The following lemma is a graded analogue of Lemma~2.8 in
\cite{ajd}. As in \cite{ajd}, it gives us a characterization of a graded strongly
nil clean element since ``being a homogeneous nilpotent'' is a graded 
ABAB-compatible property.
\begin{lemma}
Let $R$ be a $G$-graded ring and $M$ a $G$-graded right $R$-module. 
Also, let $\mathcal{P}$ be a graded ABAB-compatible property. An 
endomorphism $\varphi\in\mathrm{END}_R(M)_e$ is then the sum of a
homogeneous idempotent element $\epsilon$ and a homogeneous element $\alpha$
which has property $\mathcal{P}$ such that $\alpha\epsilon=\epsilon\alpha$ if
and only if there exists a direct sum decomposition $M=A\oplus B,$ where
$A, B\in R-gr,$ such that $\varphi|_A$ is an element of $\mathrm{END}_R(A)$
with property $\mathcal{P}$ and $(1-\varphi)|_B$ is an element of 
$\mathrm{END}_R(B)$ with property $\mathcal{P}.$ In case 
$\varphi\in\mathrm{END}_R(M)_g$ has property $\mathcal{P},$ where $g\neq e,$
the decomposition of $M$ is trivial.
\end{lemma}
\begin{proof}
Let $\varphi=\epsilon+\alpha,$ with an idempotent element $\epsilon$ and an
element $\alpha$ with property $\mathcal{P},$ where
$\epsilon, \alpha\in\mathrm{END}_R(M)_e,$ and also, let us assume that 
$\epsilon\alpha=\alpha\epsilon.$ If $1$
denotes the identical mapping, then $1-\epsilon\in\mathrm{END}_R(M)_e,$ and let
us set $A=(1-\epsilon)(M)$ and $B=\epsilon(M).$ Then $A, B\in R-gr$ and 
$M=A\oplus B.$ Now, like in the proof of Lemma~2.8 in 
\cite{ajd}, one may verify that $\varphi|_A$ and $(1-\varphi)|_B$ have the 
desired properties.\\
\indent In the other direction, let $\epsilon$ be the
projection onto $B$ with kernel $A.$ Then $\epsilon\in\mathrm{END}_R(M)_e$ and
the rest goes as in the proof of Lemma~2.8 in \cite{ajd}.\\
\indent The second assertion is obvious. 
\end{proof}
\begin{corollary}
Let $R$ be a $G$-graded ring and $M\in R-gr.$ An element 
$\varphi\in\mathrm{END}_R(M)_e$ is graded strongly nil clean if and only if there
exists a direct sum decomposition $M=A\oplus B,$ where $A, B\in R-gr,$ such that
$A$ and $B$ are $\varphi$-invariant and such that 
$\varphi|_A\in\mathrm{END}_R(A)$ and $(1-\varphi)|_B\in\mathrm{END}_R(B)$
are nilpotent. In case $\varphi\in\mathrm{END}_R(M)_g$ is nilpotent, where 
$g\neq e,$ the decomposition of $M$ is trivial.
\end{corollary}
\indent Of course, as ``being a unit'' is ABAB-compatible property (see
\cite{ajd}), ``being a homogeneous unit'' is a graded ABAB-compatible property.
Therefore, we also introduce the following notion.
\begin{definition}
A homogeneous element of a $G$-graded ring $R$ is said to be \emph{graded clean} 
if it can be written as a sum of a homogeneous idempotent and a homogeneous unit. 
A $G$-graded ring is said to be \emph{graded clean} if every of its homogeneous 
elements is graded clean.
\end{definition}
\begin{remark}
If $R=\bigoplus_{g\in G}R_g$ is a graded clean ring, then we obviously have that
$R_e$ is a clean ring, and that every nonzero homogeneous element not coming
from $R_e$ is a unit. Also, unlike the classical case (see \cite{ajd}), a graded nil
clean ring does not have to be a graded clean ring. Namely, it is enough to look at
Example~\ref{example}. Obviously, every graded division ring, that is, a graded ring
in which every homogeneous element is invertible, is graded clean.
\end{remark}
\begin{remark}
We believe that this is the appropriate place to put some ideas about some other
notions one might want to study. Inspired by the notion of a 2-clean element from
\cite{xt}, if $R$ is a $G$-graded ring,
we may define an element $a\in R$ to be \emph{graded 2-nil-clean} if it can be
written as a sum of a homogeneous idempotent element and two homogeneous
nilpotent elements. Notice that we do not require a graded 2-nil-clean element to be
homogeneous. A $G$-graded ring $R$ is then said to be \emph{graded 2-nil-clean} if
every of its elements is graded 2-nil-clean. As an example, we have a 
$\mathbb{Z}$-graded ring $R=M_2(S),$
with $R_0=\left(%
\begin{array}{cc}
  S & 0 \\
  0 & S \\
\end{array}%
\right),$ $R_1=\left(%
\begin{array}{cc}
  0 & S \\
  0 & 0 \\
\end{array}%
\right),$ $R_{-1}=\left(%
\begin{array}{cc}
  0 & 0 \\
  S & 0 \\
\end{array}%
\right),$ $R_i=\left(%
\begin{array}{cc}
  0 & 0 \\
  0 & 0 \\
\end{array}%
\right)$ for $i\in\mathbb{Z}\setminus\{-1,0,1\},$ where $S$ is a Boolean
ring.\\
\indent Also, in definition of a graded nil clean element one may also discard 
the assumption
of being a homogeneous element. That would also transfer to the notion of a
graded nil clean ring of course. If $S$ is a Boolean ring, and $C_2=\{e,g\}$ a cyclic
group of order two, one such example is a $C_2$-graded ring
$R=\left(%
\begin{array}{cc}
  S & S \\
  0 & S \\
\end{array}%
\right),$ with grading 
$R_e=\left(%
\begin{array}{cc}
  S & 0 \\
  0 & S \\
\end{array}%
\right),$
$R_g=\left(%
\begin{array}{cc}
  0 & S \\
  0 & 0 \\
\end{array}%
\right).$
\end{remark}
\indent The next result represents a graded version of Proposition~3.15 in
\cite{ajd}.
\begin{lemma}\label{lemma2}
Let $R=\bigoplus_{g\in G}R_g$ be a $G$-graded ring and $I$ a homogeneous ideal
of $R$ which is graded-nil. Then $R$ is graded nil clean if and only if $R/I$ is
graded nil clean.
\end{lemma}
\begin{proof}
Suppose that $R$ is graded nil clean. Then $R/I$ is a homomorphic image of $R$
by a degree-preserving homomorphism, and therefore $R/I$ is graded nil clean.\\
\indent Now, let $R/I$ be a graded nil clean ring. We need to show that $R_e$ is 
a nil clean
ring and that every homogeneous element not coming from $R_e$ is nilpotent. We
follow the proof of Proposition~3.15 in \cite{ajd} in order to prove that $R_e$ is nil
clean. So, let us assume that $r\in R_e.$ Since $\bar{R}=R/I$ is graded nil clean, 
$\bar{r}\in(R/I)_e=R_e/I\cap R_e$ can be written as a sum of an idempotent 
element $\bar{f}\in(R/I)_e$ and a nilpotent element $\bar{b}\in(R/I)_e.$ Now, if we
apply Proposition~27.1 in \cite{af} (see also \cite{tyl}) to $R_e,$ we see that a
homogeneous idempotent element modulo a homogeneous ideal which is graded-nil
can be lifted to a homogeneous idempotent of $R.$ Hence, 
$\bar{f}$ can be lifted to a homogeneous idempotent $f\in R_e.$ Now, 
$\overline{r-f}$ is a nilpotent element from $(R/I)_e.$ Since $I$ is graded-nil, it
follows that $r-f$ is a nilpotent element in $R_e.$ Therefore, $R_e$ is nil clean. Now
assume that $r\in R_g,$ where $g\neq e.$ Then $\bar{r}\in(R/I)_g$ is nilpotent
since $R/I$ is graded nil clean. According to assumption, $I$ is graded-nil, and so, 
$r$ is nilpotent, which completes the proof. 
\end{proof}
\indent It is proved in \cite{ajd} that a ring $R$ is nil clean if and only if 
$J(R)$ is nil and $R/J(R)$ is nil clean. We establish a similar structure result 
for $G$-graded rings.
\begin{lemma}
Let $R=\bigoplus_{g\in G}R_g$ be a $G$-graded ring which is graded nil clean. If 
$G$ is finite, then $J^g(R)$ is graded-nil.
\end{lemma}
\begin{proof}
Since $R$ is graded nil clean, we have that $R_e$ is nil clean. According to
Proposition~3.16 in \cite{ajd} applied to the ring $R_e,$ we have that 
$J(R_e)$ is nil. Now, Corollary 4.2 in \cite{cm} implies that
$J(R_e)=J^g(R)\cap R_e.$ Therefore, if $a$ is an element of the $e$-th component
of $J^g(R),$ it is nilpotent. On the other hand, if $a\in J^g(R)_g,$ where $g\neq e,$
then $a\in R_g,$ and hence $a$ is nilpotent. Therefore $J^g(R)$ is graded-nil. 
\end{proof}
\begin{corollary}
Let $R$ be a $G$-graded ring and $G$ finite. Then $R$ is graded nil clean if and only
if $J^g(R)$ is graded-nil and $R/J^g(R)$ is graded nil clean.
\end{corollary}
\begin{remark}
Of course, by Lemma~\ref{lemma2}, if $J^g(R)$ is graded-nil and $R/J^g(R)$ is
graded nil clean, then $R$ is graded nil clean for any cardinality of $G.$
\end{remark}
\indent We would also like to have a graded analogue of Corollary~3.22 in
\cite{ajd}. In order to obtain it, let us introduce the notion of a \emph{graded
strongly $\pi$-regular element} and inspect a relationship between such an element
and a graded strongly nil clean element.
\begin{definition}
A homogeneous element $a$ of a $G$-graded ring is said to be
\emph{graded strongly $\pi$-regular} if it can be written as a sum of a
homogeneous idempotent element $f$ and a homogeneous
unit $u$ such that $fa=af$ and $faf$ is nilpotent.
\end{definition}
\indent Notice that, as in the classical case, the uniqueness of a graded strongly 
$\pi$-regular decomposition holds in a $G$-graded ring 
$R=\bigoplus_{g\in G}R_g.$
First assume that a homogeneous element $a$ is a unit. Then the assertion is clear.
Now assume that $a\in R_e$ and $f+u$ and $f'+v$ are both strongly $\pi$-regular
decompositions of $a.$ It clearly suffices to show that $f = f'.$ However, this 
follows directly from Proposition~2.6 in \cite{ajd} applied to $R_e.$\\
\indent Of course, every
graded strongly nil clean element is strongly nil clean element and hence a 
strongly $\pi$-regular element according to Proposition~3.5 in \cite{ajd}.
\begin{lemma}\label{nilpotent}
Let $R=\bigoplus_{g\in G}R_g$ be a $G$-graded ring and let $a\in R$ be a graded
strongly $\pi$-regular element with graded strongly $\pi$-regular decomposition
$a=f+u.$ Then $a$ is graded strongly nil clean element if and only if $2f-1+u$ is
nilpotent and $u\in R_e.$ 
\end{lemma} 
\begin{proof}
Let $a$ be a graded strongly nil clean element. If $0\neq a\in R_g,$ where 
$g\neq e,$ then $f=0$ and $u=a$ is a unit. Since $a$ is by assumption graded
strongly nil clean element, we have that $u$ is also a nilpotent element, which is
impossible. Therefore, $a\in R_e.$ Then $a=f'+b$ for some 
idempotent element $f'\in R_e$ and a nilpotent element $b\in R_e$ which
commutes with $f'.$ Also, $a=(1-f')+(2f'-1+b)$ is a strongly $\pi$-regular
decomposition of $a$ in $R_e$ (see Proposition~3.5 in \cite{ajd}). Since strongly 
$\pi$-regular decomposition is unique according to Proposition~2.6 in \cite{ajd}, we
have that $f=1-f'$ and $u=2f'-1+b.$ However, we then have that $2f-1+u=b$ is a
nilpotent element, and, of course, $u\in R_e.$\\
\indent The converse follows from Proposition~3.9 in \cite{ajd} applied to the ring
$R_e.$ 
\end{proof}
\indent With this in mind, a graded version of Theorem~3.21 in \cite{ajd} can be
proved which implies a graded version of Corollary~3.22 in \cite{ajd}.
\begin{theorem}
Let $R=\bigoplus_{g\in G}R_g$ be a $G$-graded ring and $I$ a homogeneous
nilpotent ideal of $R.$ If $a$ is a homogeneous element of $R$ such that $\bar{a}$
is a graded strongly nil clean element in $\bar{R}=R/I,$ then $a$ is a graded 
strongly nil clean element in $R.$
\end{theorem}
\begin{proof}
First assume that $\bar{a}\in(R/I)_e.$ Since $\bar{a}$ is a graded strongly nil clean
element in $R/I,$ it can be  written as a sum of an idempotent element 
$\bar{f}\in(R/I)_e$ and a nilpotent element $\bar{b}\in(R/I)_e$ such that
$\bar{f}\,\bar{b}=\bar{b}\bar{f}.$ Every strongly nil clean element is strongly 
$\pi$-regular, and $\overline{1-f}+\overline{2f-1+b}$ is a strongly $\pi$-regular
decomposition of $\bar{a}$ in $(R/I)_e.$  Following the proof of Theorem~3.21 in
\cite{ajd}, there exists an idempotent $f'\in R_e$ and a unit $u\in R_e$ such that
$a=f'+u$ is a strongly $\pi$-regular decomposition of $a$ in $R_e.$ It is enough to
show that $2f'-1+u$ is nilpotent by Lemma~\ref{nilpotent}. By using equalities 
$\bar{f'}=\overline{1-f}$ and $\bar{u}=\overline{2f-1+b}$ one gets 
$\overline{2f' -1+u}=\bar{b}.$ Then $2f' -1+u$ is nilpotent by the nilpotency of 
$\bar{b}$ and $I.$ Now assume that $\bar{a}\in (R/I )_g,$ where $g\neq e.$ Then
$\bar{a}$ is a nilpotent element. But the assumption on $I$ implies that $a$ is
nilpotent in $R,$ which completes the proof.
\end{proof}
\begin{corollary}\label{corollary}
Let $R$ be a $G$-graded ring and $I$ a homogeneous nilpotent ideal of $R.$ Then
$R$ is graded strongly nil clean if and only if $R/I$ is graded strongly nil clean.
\end{corollary}
\indent We now turn our attention to matrix rings.\\
\indent In \cite{hn} it is proved that a matrix ring over a clean ring is also clean.
We do not have such a result in the
$G$-graded setting. However, we have a graded version of Theorem in \cite{hn}.
\begin{theorem}\label{theoremm}
Let $R=\bigoplus_{g\in G}R_g$ be a $G$-graded ring and let $1=f_1+\dots+f_n$
in $R,$ where the $f_i$ are orthogonal idempotents from $R_e.$ If each 
$f_iRf_i$ is graded clean, and if $R$ has no nonzero homogeneous zero divisors, 
then $R$ is graded clean.
\end{theorem}
\indent We first prove the following lemma which represents a graded version of
Lemma in \cite{hn}.
\begin{lemma}
Let $f$ be a homogeneous idempotent of a $G$-graded ring 
$R=\bigoplus_{g\in G}R_g$ and let $\bar{f}=1-f.$ Let us assume that $fRf$ and 
$\bar{f}R\bar{f}$ are graded clean rings. If $R$ has no nonzero 
homogeneous zero divisors, then $R$ is a graded clean ring.
\end{lemma}
\begin{proof}
We observe
$R$ in its Peirce decomposition $R=\left(%
\begin{array}{cc}
  fRf & fR\bar{f} \\
  \bar{f}Rf & \bar{f}R\bar{f} \\
\end{array}%
\right).$ According to Lemma in \cite{hn}, $R_e$ is clean. Now,
let $0\neq A\in R_g,$ where $g\neq e,$ and let $A=\left(%
\begin{array}{cc}
  a & x \\
  y & b \\
\end{array}%
\right).$ Then $a,x,y,b\in R_g$ and $a,b$ are units. Let $a_1\in R_{g^{-1}}$ be 
an inverse of $a.$ Then 
$b-ya_1x\in\bar{f}R\bar{f},$ and $b-ya_1x\in R_g.$ Hence, if $b\neq ya_1x,$ we
have that $b-ya_1x=v$ is a unit.
Therefore $\left(%
\begin{array}{cc}
  a & x \\
  y & b \\
\end{array}%
\right)=\left(%
\begin{array}{cc}
  a & x \\
  y & v+ya_1x \\
\end{array}%
\right)$ is a unit, as it can be shown as in Lemma in \cite{hn} for the classical 
case, since $ya_1$ and $a_1x$ both belong to $R_e.$ Namely, $\left(%
\begin{array}{cc}
  f & 0 \\
  -ya_1 & \bar{f} \\
\end{array}%
\right),$ $\left(%
\begin{array}{cc}
  f & -a_1x \\
  0 & \bar{f} \\
\end{array}%
\right)$ and $\left(%
\begin{array}{cc}
  a & 0 \\
  0 & v \\
\end{array}%
\right)$ are units, and, on the other hand, we have $\left(%
\begin{array}{cc}
  f & 0 \\
  -ya_1 & \bar{f} \\
\end{array}%
\right)\left(%
\begin{array}{cc}
  a & x \\
  y & v+ya_1x \\
\end{array}%
\right)\left(%
\begin{array}{cc}
  f & -a_1x \\
  0 & \bar{f} \\
\end{array}%
\right)=\left(%
\begin{array}{cc}
  a & 0 \\
  0 & v \\
\end{array}%
\right).$ The case $v=0$ cannot occur since by
assumption $R$ has no nonzero homogeneous zero divisors.   
\end{proof}
\begin{proof}[Proof of Theorem \ref{theoremm}]
The assertion follows from the previous lemma by using mathematical induction 
just as in the case of Theorem in \cite{hn}. 
\end{proof}
\indent When it comes to graded nil clean rings, we establish a graded analogue
of an extension of nil clean rings to triangular matrix rings from \cite{ajd}.
\begin{theorem}
Let $R$ be a $G$-graded ring and $n$ a natural number. Then $R$ is 
graded (strongly) nil clean if and only if 
$T_n(R)(\overline{\sigma})$ is graded (strongly) nil clean triangular 
matrix ring for every $\overline{\sigma}\in G^n.$
\end{theorem}
\begin{proof}
Let $\overline{\sigma}\in G^n$ and 
$S=T_n(R)(\overline{\sigma}).$ Like in
the proof of Theorem 4.1 in \cite{ajd}, let us observe the ideal $I$ of $S$ which
consists of matrices of $S$ with zeroes along the main diagonal. Ideal $I$ is a
homogeneous nilpotent ideal of $S$ and, as in the classical case, it can be proved
that $S/I$ is isomorphic to the direct product of $n$ copies of $R.$ Since $I$ is in 
particular graded-nil, the assertion for graded nil cleanness follows from 
Lemma~\ref{lemma2} and the fact that a product of finitely many graded nil clean
rings is again a graded nil clean ring (category of graded rings is closed for finite
products according to \cite{noy}). The statement for graded strongly nil cleanness
follows from Corollary~\ref{corollary} and the fact that a product of finitely many
graded strongly nil clean rings is again a graded strongly nil clean ring.
\end{proof}
\indent In \cite{stz}, the nil cleanness of group rings is investigated.
We would like to do similar in group graded setting.\\ 
\indent We start with a simple lemma which is also of an independent interest.
\begin{lemma}\label{2nilpotent}
If a $G$-graded ring $R=\bigoplus_{g\in G}R_g$ is a graded nil clean ring, then 
$2$ is nilpotent.
\end{lemma}
\begin{proof}
Since $R$ is a graded nil clean ring, we have that $R_e$ is nil clean. Since 
$1\in R_e,$ the statement follows directly from Proposition~3.14 in \cite{ajd}.
\end{proof}
As is known, if $G$ is a group, and $H$ a normal subgroup of $G,$ then a
$G$-graded ring $R=\bigoplus_{g\in G}R_g$ can be viewed as a $G/H$-graded
ring with respect to grading $\bigoplus_{C\in G/H}R_C,$ where 
$R_C=\bigoplus_{x\in C}R_x.$
\begin{theorem}
Let $G$ be a locally $2$-finite group and $H$ a normal subgroup of $G.$ Also,
assume that $R=\bigoplus_{g\in G}R_g$ is a $G$-graded ring which is graded nil
clean as a $G/H$-graded ring. Then the $G/H$-graded group ring $R[H]$
is graded nil clean.
\end{theorem}
\begin{proof}
Following the proof of Theorem~2.3 in \cite{stz}, we may assume that $H$ is a
finite $2$-group. According to \cite{noy}, page 180, the augmentation mapping
$R[H]\to R,$ given by $\sum_{h\in H}r^hh\mapsto\sum_{h\in H}r^h,$ where 
$R[H]$ is considered as a $G/H$-graded ring, is 
degree-preserving. Therefore, the kernel of the augmentation mapping, that is, the 
augmentation ideal $\Delta(R[H]),$ is homogeneous. This means that 
$R[H]/\Delta(R[H])$ is a $G/H$-graded ring.
Moreover, $R[H]/\Delta(R[H])$ and $R$ are isomorphic as $G/H$-graded rings.
Hence $R[H]/\Delta(R[H])$ is
graded nil clean. Now, $2$ is nilpotent by Lemma~\ref{2nilpotent}. 
Theorem~9 in \cite{connel} tells us that $\Delta(R[H])$ is
nilpotent, and therefore, graded nil. Applying Lemma~\ref{lemma2} completes 
the proof.
\end{proof}
\begin{theorem}\label{theoremg}
Let $R=\bigoplus_{g\in G}R_g$ be a $G$-graded ring which has only homogeneous 
idempotent and nilpotent elements. If $R[G]$ is
graded nil clean, then $R$ is graded nil clean.
\end{theorem}
\begin{proof}
Since $R[G]$ is graded nil clean, $(R[G])_e$ is nil clean.
According to Proposition~2.1(4) in \cite{cn}, the mapping $f:R\to(R[G])_e,$
$f(\sum_{g\in G}r_g)=\sum_{g\in G}r_gg^{-1},$ is a ring
isomorphism. Therefore, $R$ is nil clean and hence graded nil clean.
\end{proof}
\begin{remark}
If $R=\bigoplus_{g\in G}R_g$ is a $G$-graded ring such that $R[G]$ has only 
homogeneous idempotent and nilpotent elements, $G$ a locally 2-finite group, 
and $(R[G])_e$ is nil clean, then $R[G]$ is a graded nil clean ring. Namely, this is 
a corollary to Theorem~2.3 in \cite{stz} since $(R[G])_e\cong R.$
\end{remark}
\indent Previous remark yields an interesting question of what can be said of 
the following implication:
\begin{equation}\label{implication}
R_e\ \textrm{is nil clean} \Rightarrow R=\bigoplus_{g\in G}R_g\ \textrm{is
graded nil clean}.
\end{equation}
\indent The following example proves that the above implication does not hold 
in general.
\begin{example}
Let $S$ be a Boolean ring, $G=\{e,g\}$ a cyclic group of order $2,$ and 
$R=\left(%
\begin{array}{cc}
  S & S \\
  S & S \\
\end{array}%
\right).$ Then $R=\left(%
\begin{array}{cc}
  S & 0 \\
  0 & S \\
\end{array}%
\right)\oplus\left(%
\begin{array}{cc}
  0 & S \\
  S & 0 \\
\end{array}%
\right)$ is a $G$-graded ring whose $e$-th component $R_e=\left(%
\begin{array}{cc}
  S & 0 \\
  0 & S \\
\end{array}%
\right)$ is a nil clean ring, but $R$ is not a graded nil clean ring since not 
all elements of $R_g=\left(%
\begin{array}{cc}
  0 & S \\
  S & 0 \\
\end{array}%
\right)$ are nilpotent.
\end{example}
\indent We continue by giving some sufficient conditions for the above 
implication to be true.
\begin{theorem}
Let $R=\bigoplus_{g\in G}R_g$ be a $G$-graded PI ring which is Jacobson radical.
Then, if $R_e$ is nil clean, $R$ is graded nil clean.
\end{theorem}
\begin{proof}
According to Proposition 3.16 in \cite{ajd}, if $S$ is a nil clean ring, then
$J(S)$ is nil. Therefore, our assumption yields that 
$J(R_e)$ is nil. Now, Theorem~3 in \cite{ko} tells 
us that $J(R)$ is nil since $R$ is by assumption PI. However, $R$ is by assumption
Jacobson radical ring. Therefore $R=J(R)$ is a nil ring. In particular, every 
homogeneous element is nilpotent. Hence $R$ is a graded nil clean ring.
\end{proof}
\begin{theorem}\label{gradedR}
Let $R=\bigoplus_{g\in G}R_g$ be a $G$-graded PI ring which is graded local, that
is, it has a unique maximal homogeneous right ideal, and let $G$ be a finite group
such that the order of $G$ is a unit in $R.$ If $R_gR_{g^{-1}}=0$ for every 
$g\in G\setminus\{e\},$ then, if $R_e$ is nil clean, $R$ is graded nil clean.
\end{theorem}
\begin{proof}
We know from Corollary 3.17 in \cite{ajd} that a ring $A$ is nil clean if and only if
$J(A)$ is nil and $A/J(A)$ is nil clean. Therefore, our assumption yields that 
$J(R_e)$ is nil and that $R_e/J(R_e)$ is nil clean. $R$ is by assumption PI, and 
hence, Theorem~3 in \cite{ko} implies that $J(R)$ is nil. Since $G$ is finite, we have 
that $J(R_e)=J^g(R)\cap R_e,$ according to Corollary~4.2 in \cite{cm}. Also, our
assumption on the order of $G$ implies that $J(R)$ is homogeneous and
$J^g(R)=J(R)$ (see Theorem~4.4 in \cite{cm}). Therefore, $R/J(R)$
is a $G$-graded ring. Since $R$ is a graded local ring, we have that $R/J(R)$ is a
graded division ring. If $H_R$ is the homogeneous part of $R,$ let 
$H_{R/J(R)}=\bigcup_{g\in G}R_g/(J(R)\cap R_g)=H_R/J(R)\cap H_R$ be the
homogeneous part of $R/J(R)$ with induced partial addition and everywhere defined 
multiplication, that is, the corresponding anneid. Then $H_{R/J(R)}$ is a simple 
anneid, that is, it has no nontrivial ideals. Let $f:H_{R/J(R)}\to R_e/J(R_e)$
be the mapping defined by $f(x+J(R)\cap H_R)=x+J(R_e)$ if $x\in R_e$ and
$f(x+J(R)\cap H_R)=0+J(R_e)$ if $x\notin R_e$ or $x\in J(R),$ and where
$x+J(R)\cap H_R\in H_{R/J(R)}$ (see also the proof of Theorem~3.2 in \cite{eig3}). 
It is easily seen that $f$ is well defined and that it is a surjective homomorphism of 
anneids. Also, since $J(R)$ is proper and $1\in R_e,$ we have that 
$R_e\neq J(R_e).$ Therefore $\ker f=0.$
Hence $H_{R/J(R)}\cong R_e/J(R_e).$ It follows that every homogeneous element
from $R/J(R)$ is graded nil clean. Hence, $R/J(R)$ is graded nil clean. Finally, 
according to Lemma~\ref{lemma2}, $R$ is graded nil clean. 
\end{proof}
\indent We conclude this paper by observing the implication
\eqref{implication} in the case of rings graded in the sense of 
Definition~\ref{sgris}.\\
\indent Of course, while observing \eqref{implication} in the case of 
$S$-graded rings inducing $S,$ letter $e$ would stand for an idempotent element
of $S.$ Definition of a \emph{graded nil clean element} of an $S$-graded ring
inducing $S$ as well as of a \emph{graded nil clean ring} is the same as in the case
of a group graded ring. However, $S$ may
have more than one nonzero idempotent. Consequently, components of a 
graded nil clean ring corresponding
to these idempotents are all nil clean rings. Of course, here we also have that
homogeneous elements, not belonging to components which correspond to
nonzero idempotent elements of $S,$ are nilpotent.\\
\indent In the proof of the next theorem we use notions of the 
\emph{graded Jacobson} $J^g(R)$
and the \emph{large graded Jacobson radical} $J^g_l(R)$ of an $S$-graded ring $R$
inducing $S,$ which are introduced in \cite{ha1}. We do not recall these notions here
since we only need their properties. For more information on these and related
radicals, one may also consult \cite{eig1,eig2,eig3}.\\
\indent In what follows, we assume that all $S$-graded rings inducing $S$ have an
identity. If $S$ is cancellative, then, according to \cite{ha}, the number of nonzero
idempotents of $S$ is finite, all components corresponding to these idempotents
have an identity, and an identity of the whole ring is a sum of identities of the
aforementioned components.\\
\indent With this in mind, we record the following characterization of $S$-graded nil
clean rings inducing $S.$  
\begin{lemma}\label{lemma2a}
Let $S$ be a cancellative partial groupoid, $R=\bigoplus_{s\in S}R_s$ an 
$S$-graded ring inducing $S$ and $I$ a
homogeneous ideal of $R$ which is graded-nil. Then $R$ is graded nil clean if 
and only if $R/I$ is graded nil clean.
\end{lemma}
\begin{proof}
If $R$ is graded nil clean then $R/I$ is graded nil clean since it is a 
homomorphic
image of $R$ taken by a homogeneous homomorphism.\\
\indent Now, let $R/I$ be a graded nil clean ring. Let $r$ be a homogeneous 
element of $R.$
Assume first that $r\in R_e,$ where $e$ is an arbitrary nonzero idempotent 
element
of $S.$ As in the proof of 
Lemma~\ref{lemma2}, one obtains that $R_e$ is nil clean.  Again the case of 
$r\in R_g,$ where  $g\neq e,$ is dealt easily just as in the proof of 
Lemma~\ref{lemma2}.
\end{proof}
\begin{theorem}
Let $S$ be a finite cancellative partial groupoid, and $R=\bigoplus_{s\in 
S}R_s$ an $S$-graded ring inducing $S$ which is also PI and Jacobson radical 
ring. If $S$ has exactly one nonzero idempotent element $e,$ and if $R_e$ is 
nil clean, then $R$ is graded nil clean. 
\end{theorem}
\begin{proof}
We first note that $J(R_e)$ is nil according to Proposition~3.16 in \cite{ajd}. 
Since $R$ is Jacobson radical ring, we have that $J(R)$ is homogeneous. Therefore
$J(R)$ coincides with the large graded Jacobson radical $J^g_l(R).$ Namely,
according to \cite{ha1}, the largest homogeneous ideal of $R$ contained in $J(R)$
coincides with $J^g_l(R).$ On the other hand, $J^g_l(R)\subseteq J^g(R)$ (see
\cite{ha1}), which implies that $J^g(R)=J^g_l(R)=R.$ Now, since all the assumptions
of Theorem~12 in \cite{eig} are satisfied, we have that $J(R)$ is nil. Therefore
$R=J(R)$ is nil and hence graded nil clean.
\end{proof}
\begin{theorem}
Let $S$ be a finite cancellative partial groupoid, $F$ a field with 
$\mathrm{char}(F)=0$ or $\mathrm{char}(F)>|S|,$ and let 
$R=\bigoplus_{s\in S}R_s$ be an $S$-graded $F$-algebra inducing $S$ which is 
also PI. Assume also that $R$ is graded local ring, that is, that it has a unique 
maximal homogeneous right ideal. If $S$ has
exactly one nonzero idempotent element $e$ such that 
$st=e\Rightarrow s=e\vee t=e$ $(s,t\in S),$ if $J^g(R)=J^g_l(R),$ and if $R_e$ is 
nil clean, then $R$ is graded nil clean. 
\end{theorem}
\begin{proof}
We first note that $J(R_e)$ is nil and  $R_e/J(R_e)$ is nil clean according to
Corollary~3.17 in \cite{ajd}. Now, $J(R)$ is nil according to Theorem~12 in
\cite{eig}. Also, $J(R)$ is homogeneous by Corollary~4 in \cite{ak2}. Therefore
$J(R)=J^g(R)$ since $J^g_l(R)$ coincides with the largest homogeneous ideal of 
$R$ contained in $J(R)$ (see \cite{ha1}). Hence $R/J(R)$ is an $S$-graded ring
inducing $S.$ Also, since $J(R_e)=J^g(R)\cap R_e$ holds, according to 
\cite{ha1}, the $e$-th component of $R/J(R)$ is $R_e/J(R_e).$ The rest of the 
proof goes as in the proof of Theorem~\ref{gradedR} with the help of 
Lemma~\ref{lemma2a}.
\end{proof}
\section*{Acknowledgements} The authors would like to express their sincere
gratitude to Professor Ivan Shestakov and to the referee for handling this manuscript.
The second author was supported by TUBITAK (No. 117F070). 

\flushleft\small{Emil Ili\'{c}-Georgijevi\'{c}\\
University of Sarajevo\\Faculty of Civil Engineering\\
Patriotske lige 30, 71000 Sarajevo, Bosnia and Herzegovina\\
e-mail: emil.ilic.georgijevic@gmail.com\\
\flushleft\small{Serap \c{S}ahinkaya}\\
Gebze Technical University\\Department of Mathematics\\
Gebze/Kocaeli, Turkey\\
e-mail: srpsahinkaya@gmail.com
\end{document}